\documentclass[letterpaper, 10 pt, conference]{ieeeconf}  
\IEEEoverridecommandlockouts                           
\usepackage{amsmath} 
\usepackage{amssymb} 
\newcommand {\matr}[2]{\left[\begin{array}{#1}#2\end{array}\right]}
\newcommand {\stack}[2]{\begin{array}{#1}#2\end{array}}

\usepackage{algorithm}
\usepackage[noend]{algpseudocode}
\usepackage{graphicx,psfrag}

\usepackage{tabularx,booktabs,multirow}
\usepackage{color}
\definecolor{wheat}{rgb}{0.96,0.87,0.70}
\definecolor{mario}{rgb}{0.8,0.8,1}
\definecolor{seb}{rgb}{0.8,1,0.8}
\definecolor{robert}{rgb}{1,0.8,0.8}
\definecolor{paolo}{rgb}{0.8,0.8,0.8}

\usepackage{tikz}

\usepackage{accents}
\usepackage{expl3,xparse}

\newcommand{\pout}[1]{p^\mathrm{out}_{#1}}
\newcommand{\pin}[1]{p^\mathrm{in}_{#1}}

\newcommand{\T}{}
\newcommand{\Tin}{}
\newcommand{\Tout}{}
\newcommand{\Toutm}{}
\newcommand{\UT}{}
\newcommand{\UTin}{}
\newcommand{\UTout}{}
\RenewDocumentCommand\T{mg}{\ensuremath{t_{#1 \IfNoValueTF{#2}{}{,#2} }}}
\RenewDocumentCommand\Tin{mg}{\ensuremath{{\T{#1}{#2}^\mathrm{in}}}}
\RenewDocumentCommand\Tout{mg}{\ensuremath{ {t^\mathrm{out}_{#1 \IfNoValueTF{#2}{}{,#2} }} }}
\RenewDocumentCommand\Toutm{mg}{\ensuremath{ {t^\mathrm{out-}_{#1 \IfNoValueTF{#2}{}{,#2} }} }}
\RenewDocumentCommand\UT{mg}{\ensuremath{u^\mathrm{t}_{#1 \IfNoValueTF{#2}{}{,#2} }}}
\RenewDocumentCommand\UTin{mg}{\ensuremath{u^\mathrm{t,in}_{#1 \IfNoValueTF{#2}{}{,#2} }}}
\RenewDocumentCommand\UTout{mg}{\ensuremath{u^\mathrm{t,out}_{#1 \IfNoValueTF{#2}{}{,#2} }}}

\newcommand{\ToutUB}{}
\RenewDocumentCommand\ToutUB{mg}{\ensuremath{ {t^\mathrm{out,ub}_{#1 \IfNoValueTF{#2}{}{,#2} }} }}
\newcommand{\ToutLB}{}
\RenewDocumentCommand\ToutLB{mg}{\ensuremath{ {t^\mathrm{o,lb}_{#1 \IfNoValueTF{#2}{}{,#2} }} }}

\renewcommand{\P}[0]{\mathcal{P}}

\newcommand{\V}{V}
\newcommand{\feasT}{\mathcal{T}}

\ExplSyntaxOn
\DeclareExpandableDocumentCommand{\IfNoValueOrEmptyTF}{mmm}
{
	\IfNoValueTF{#1}{#2}
	{
		\tl_if_empty:nTF {#1} {#2} {#3}
	}
}
\ExplSyntaxOff
\NewDocumentCommand\initstate{mg}{\ensuremath{ \hat{x}_{ \IfNoValueOrEmptyTF{#1}{}{#1,}0 } }}

\NewDocumentCommand\vars{mg}{\ensuremath{z_{#1 \IfNoValueTF{#2}{}{,#2} }}}
\NewDocumentCommand\state{mg}{\ensuremath{x_{#1 \IfNoValueTF{#2}{}{,#2} }}}
\NewDocumentCommand\barstate{mg}{\ensuremath{\bar{x}_{#1 \IfNoValueTF{#2}{}{,#2} }}}
\NewDocumentCommand\control{mg}{\ensuremath{u_{#1 \IfNoValueTF{#2}{}{,#2} }}}

\newtheorem{Proposition}{Propositon}

\title{\LARGE \bf
A Feasibility-Enforcing Primal-Decomposition SQP Algorithm for Optimal Vehicle Coordination
}

\author{Mario Zanon, Robert Hult, S\'ebastien Gros and Paolo Falcone
\thanks{This work was supported by Copplar (project number 32226302), the Swedish Research Council (VR, grant number 2012-4038) and the European Commission Seventh Framework under the AdaptIVe (grant number 610428). The authors are with the Department of Signals and Systems, Chalmers University of Technology, G\"{o}teborg, Sweden. e-mail: \{name.surname@chalmers.se\}}%
}

\begin{document}

\maketitle
\thispagestyle{empty}
\pagestyle{empty}

\begin{abstract}
	In this paper we consider the problem of coordinating autonomous vehicles approaching an intersection. We cast the problem in the distributed optimisation framework and propose an algorithm to solve it in real time. We extend previous work on the topic by testing two alternative algorithmic solutions in simulations. Moreover, we test our algorithm in experiments using real cars on a test track. The experimental results demonstrate the applicability and real-time feasibility of the algorithm.
\end{abstract}

\section{Introduction}


The problem of coordinating autonomous vehicles approaching an intersection has recently attracted increasing research interest
\cite{Azimi2011,Campos2013,Lee2012}. The main motivation behind this research topic is to use automation in order to (a) reduce the amount of accidents, (b) reduce pollution and energy consumption and (c) increase the capacity of the infrastructure.

In order to achieve these three goals, it is necessary to introduce communication between the involved agents and the infrastructure, and to design a suitable algorithmic framework in order to compute a coordination policy. The coordination problem can be framed as a distributed mixed-integer optimal control problem. While this class of problems has a high complexity and is known to be NP-hard \cite{Colombo2012}, by designing tailored optimisation algorithms and heuristics one can aim at computing approximated solutions in a rather short time. In this paper, however, we assume a prescribed crossing order and focus on the continuous part of the problem, i.e. the solution of the distributed optimal control problem. Future research will aim at extending the presented results in order to also optimise the crossing order.

The problem we address is not restricted to vehicles approaching an intersection, but can be framed more generally as a resource allocation problem subject to dynamic constraints. Such problems find applications in e.g. advanced manufacturing, traffic management, and logistics.

This paper is structured as follows. In Section~\ref{sec:formulation} we introduce the formulation of the coordination problem as a distributed optimisation one. In Section~\ref{sec:sqp} we present the distributed algorithm used to solve the problem. In Section~\ref{sec:results} we discuss the simulation and experimental results. We present the concluding remarks and outline for future research in Section~\ref{sec:conclusions}.

\begin{figure}[t]
	\begin{center}
		\includegraphics[width=0.32\textwidth]{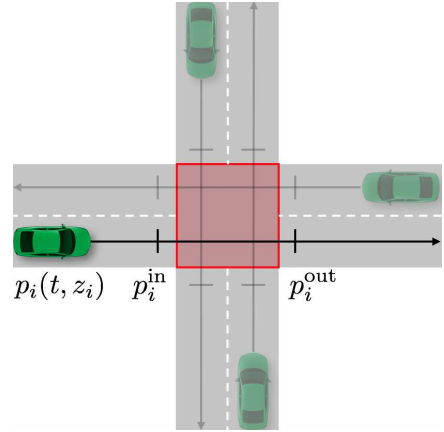}
		\caption{Illustration of the considered intersection scenario.}
		\label{fig:scenario}
	\end{center}
\end{figure}

\section{Problem Formulation}
\label{sec:formulation}
We consider the problem of $N_\mathrm{a}$ vehicles approaching an intersection, as illustrated in Figure~\ref{fig:scenario}.
For each vehicle $i$, we define the in- and out-times, i.e. the times at which the vehicle enters and exits the intersection, respectively as $\Tin{i}$ and $\Tout{i}$. We lump them together in vectors $\T{i} := (\Tin{i},\Tout{i})$ and we define the vector of all in- and out-times as $\T{} := (\T{1},\ldots,\T{N_\mathrm{a}})$.

The cost $\V_i(\T{i})$ associated with each individual vehicle is defined by an MPC problem specific to the vehicle. We define the states and controls of vehicle $i$ at time $k$ as $\state{i}{k}$ and $\control{i}{k}$ respectively. Moreover, we define vector $\vars{i} := (\vars{i}{0},\ldots,\vars{i}{N-1},x_{i,N})$ with $\vars{i}{k}:=(\state{i}{k},\control{i}{k})$ and the prediction horizon $N\in \mathbb{N}$. For simplicity, we assume linear dynamics and affine path constraints. Finally, we denote the set of all integers in a given interval as $\mathbb{I}_{[a,b]}=\{a,a+1,\ldots,b\}$.
The MPC problem of vehicle $i$ is then given by
\begin{subequations}\label{eq:localProblem}
	\begin{align}
	&&\hspace{-2em}\V_i(\T{i}):= \nonumber \\
	&&\hspace{-2em} \min_{\vars{i}} 
	& \quad J_i(\vars{i}) \\ 
	&&\hspace{-2em}\mathrm{s.t.} 			& \quad \state{i}{0} = \initstate{i} \label{eq:localIC}\\
	&&& \quad \state{i}{k+1}= A_i\state{i}{k} + B_i\control{i}{k}, & k\in\mathbb{I}_{[0,N-1]}, \label{eq:localDyn}\\
	&&& \quad D_{i,k}\state{i}{k} + E_{i,k}\control{i}{k} \leq e_{i,k}, & k\in\mathbb{I}_{[0,N-1]}, \label{eq:localIneq}\\
	&&& \quad p_i(\Tin{i},\vars{i})-\pin{i} = 0, \label{eq:localPosIn}\\
	&&& \quad p_i(\Tout{i},\vars{i} )-\pout{i} = 0, \label{eq:localPosOut}
	\end{align}
\end{subequations}
where $J_i(\vars{i})$ is a vehicle-specific quadratic cost and constraints~\eqref{eq:localPosIn} and~\eqref{eq:localPosOut} make use of function $p_i$, describing the position of vehicle $i$ on its path along road, in order to force the vehicle to enter and exit the intersection at the prescribed times, see~\cite{Hult2016}. Because times $\T{i}$ are fixed, Problem~\eqref{eq:localProblem} is a Quadratic Program (QP).
The set of feasible in- and out-times is then the domain of MPC Problem~\eqref{eq:localProblem}, i.e.
\begin{align}
\label{eq:localFeasibility}
\feasT_i := \text{dom}(\V_i(\T{i})).
\end{align}

 The coordination problem can then be formulated as 
\begin{subequations}
	\label{eq:centralProblem}
	\begin{align}
	\min_{\T{}} & \quad \sum_{i=1}^{N_\mathrm{a}} \V_i (\T{i})	\label{eq:centralCost}												\\
	\mathrm{s.t.} 			& \quad  \T{i} \in \feasT_i,	& i\in \mathbb{I}_{[1,N_\mathrm{a}]}, \label{eq:agentFeasibility}					\\
	& \quad \Tout{i+1} \leq \Tin{i}, 			& i\in \mathbb{I}_{[1,N_\mathrm{a}]}, \label{eq:precedenceConstraints}
	\end{align}
\end{subequations}
such that the sum of each vehicle's individual cost is minimised, subject to having at most one vehicle in the intersection at any time~\eqref{eq:precedenceConstraints}.
Note that, because constraints~\eqref{eq:localPosIn}-\eqref{eq:localPosOut} are generally nonlinear in $\T{i}$, Problem~\eqref{eq:centralProblem} is nonconvex.


We assume that the technical assumption \cite[Assumption 1]{Hult2016} holds and $p_i(t,\vars{i})$ is monotonically increasing in $t$ and differentiable. 
Then, $\feasT_i$ can be described by the in- out-times relative to the trajectories obtained by solving the linear programs (LPs):
\begin{subequations}
	\label{eq:LP_time}
	\begin{align}
	&\hspace{-1em}(\max)\min_{\vars{i}} \  x_N \ \mathrm{s.t.} \ \eqref{eq:localIC}-\eqref{eq:localIneq}, && \text{for bounds on} && \Tin{i}, \\
	&\hspace{-1em}(\max)\min_{\vars{i}} \  x_N \ \mathrm{s.t.} \ \eqref{eq:localIC}-\eqref{eq:localPosIn},
	\end{align}
\end{subequations}
for $\Tin{i}$-dependent constraints on $\Tout{i}$, which we denote as $\Tout{i}\leq\ToutUB{i}\left (\Tin{i} \right )$, $\Tout{i}\geq\ToutLB{i}\left (\Tin{i} \right )$.
For all details concerning this formulation we refer to~\cite{Hult2016}. 





%

\section{Numerical Solution of the NLP}
\label{sec:sqp}

In this section, we present an adaptation of sequential quadratic programming (SQP) in a distributed setting for solving Problem~\eqref{eq:centralProblem}.
We remark that,
by definition of $\V_i(\T{i})$ and $\feasT_i$, 
Problem~\eqref{eq:centralProblem} is non-smooth. 
A thorough discussion on the continuity and differentiability properties of $\V_i$ and $\feasT_i$ has been presented in~\cite{Hult2016}. 
Note however that, provided that an interior-point QP solver is used in the local problems \eqref{eq:localProblem} and \eqref{eq:LP_time}, one obtains a smooth approximation of the non-smooth problem. While this solves theoretical issues regarding the convergence of algorithms for smooth optimisation, slow convergence can in principle not be excluded, as the smoothly approximated problem becomes highly nonlinear at the points of non-smoothness of the original problem.



\subsection{Sequential Quadratic Programming}

%

For notational simplicity, we rewrite NLP~\eqref{eq:centralProblem} as
\begin{align}
\label{eq:nlp}
\min_{\T{}} \ \ & f(\T{}) &&
\mathrm{s.t.} \ \ h(\T{})\geq0,
\end{align}
where $f(\T{})=\sum_{i=1}^{N_\mathrm{a}} \V_i (\T{i})$ and we lump Constraints~\eqref{eq:agentFeasibility}-\eqref{eq:precedenceConstraints} in function $h$. We define the associated Lagrangian as $\mathcal{L}(\T{},\mu):= f(\T{}) - \mu^\top h(\T{})$.
Starting from an initial guess $v^{(0)}=(\T{}^{(0)}, \ \mu^{(0)})$, SQP iteratively computes $v^{(j)}$ using
\begin{align}\label{eq:pdstep}
v^{(j+1)} = v^{(j)} + \alpha^{(j)} \Delta v^{(j)},
\end{align}
with $\alpha^{(j)}\in(0,1]$ and 
\begin{align}
\label{eq:sqp_full_step}
\Delta v^{(j)} = (\Delta \T{}^{(j)}, \ \tilde \mu^{(j)} - \mu^{(j)}),
\end{align}
obtained as the primal-dual solution $(\Delta \T{}^{(j)}, \  \tilde \mu^{(j)})$ of the quadratic programming (QP) subproblem
\begin{subequations}
	\label{eq:qp_k_definition}
	\begin{align}
	\underset{\Delta \T{}}{\mathrm{min}} \ \ & \frac12 \Delta \T{}^\top H^{(j)} \Delta \T{} + \nabla   f(\T{}^{(j)})^\top \Delta \T{} \label{eq:qp_k_cost}\\
	\mathrm{s. t.} \ \ 
	&   h(\T{}^{(j)}) + \nabla   h(\T{}^{(j)})^\top \Delta \T{} \geq 0. \label{eq:qp_k_ineq_cstr} 
	\end{align}
\end{subequations}
The iterations are stopped when the KKT residual $r$ satisfies
\begin{align}
	r^{(j)}:= \left \| \stack{c}{  \nabla f(\T{}^{(j)}) - \nabla  h(\T{}^{(j)})\mu^{(j)} \\  \min(0, h(\T{}^{(j)})) } \right \|_\infty \leq \epsilon.
\end{align}

Existing variants of SQP differ in the computations of the step size $\alpha^{(j)}$ and the so-called Hessian matrix approximation $H^{(j)}$. If exact Hessian is used, i.e. $H^{(j)}=\nabla^2_{\T{}\T{}}{\mathcal{L}(\T{}^{(j)}, \mu ^{(j)} )}$, the KKT conditions of each QP subproblem~\eqref{eq:qp_k_definition} coincide with a special form of linearisation of the KKT conditions of NLP~\eqref{eq:nlp} evaluated at $\T{}^{(j)}$. For more details on SQP, we refer to e.g.~\cite{Nocedal2006}.

The QP matrices $\nabla   f(\T{}^{(j)})$, 
$\nabla   h(v^{(j)})$ and the Hessian of the Lagrangian $\nabla^2_{\T{}\T{}}{\mathcal{L}(\T{}^{(j)}, \mu ^{(j)} )}$ are called first and second-order sensitivities respectively, whose computation is presented in detail in~\cite{Hult2016}. Here we simply recall that the evaluations of $f,h$ require the solution of 2 LPs and 1 QP, while the fisrt and second order sensitivities 
are obtained at a marginal additional cost with respect to the function evaluations.


\subsection{Hessian Regularisation}
For nonconvex optimisation algorithms, it is in general required that the reduced Hessian is positive definite, in order to avoid solving indefinite QP subproblems. Therefore, a modification of the Lagrangian Hessian might be required, especially if the linearisation point is far from the optimum. 
For a thorough analysis we refer the interested reader to e.g.~\cite{Nocedal2006} and references therein. 

In this paper, we apply the simple strategy of removing all directions of negative curvature present in the Hessian by
adding the minimal regularisation needed to obtain a positive curvature in all directions. We do this by exploiting the block-diagonal structure of the Hessian to perform an eigenvalue decomposition of each $2$-by-$2$ block and saturate all eigenvalues to a predefined minimum positive value. Given the small size, the eigenvalue decomposition can be made very efficient. 


\subsection{Globalisation and Merit Function}

In order to guarantee global convergence of SQP algorithms, several approaches have been proposed in the literature, including \emph{linesearch} and \emph{trust region} methods, both of which rely on a so-called \emph{merit function} to measure progress towards optimality and feasibility.
Several merit functions have been proposed and we refer the interested reader to~\cite{Nocedal2006} and references therein for more details on the topic. In this paper we make use of linesearch based on the so-called $\ell_1$ merit function, defined as
\begin{align}
M(\T{}) = f(\T{}) + \sigma \| h^-(\T{}) \|_1 ,
\end{align}
where $h^-(\T{}) := \mathrm{min}(h(\T{}),0)$ and $\sigma$ is a parameter which must be chosen such that $\sigma>\| \mu \|_\infty$.

Note that one can compute the first-order term of the Taylor expansion of $M$ in the SQP direction $y$ as~\cite{Nocedal2006}:
\begin{align}
\mathrm{D}M(x)[y] y = \nabla f(x) y - \sigma \| h^-(x) \|_1,
\end{align}
where we define the directional derivative of $M$ in direction $y$ evaluated at $x$ as $\mathrm{D}M(x)[y]$.

Once the QP solution $\Delta v^{(j)}$ is obtained, linesearch globalisation techniques choose the step size so as to enforce a decrease in the merit function based on the Armijo condition
\begin{align}
\label{eq:armijo}
M( \T{} + \alpha\Delta \T{} ) \leq M( \T{} ) + \gamma \mathrm{D} M( \T{} )[\Delta \T{}] \alpha\Delta \T{},
\end{align}
where $\gamma\in(0,0.5]$ is a fixed constant.

Because solving~\eqref{eq:armijo} exactly is in general computationally expensive, a popular technique called \emph{backtracking} consists in starting with $\alpha = 1$ and iteratively reducing it using $\alpha \leftarrow \beta \alpha$ with $\beta\in(0,1)$ until condition~\eqref{eq:armijo} is satisfied.

%
%
%

\subsection{Local Feasibility Issues}

NLP solvers only guarantee feasibility of the constraints at convergence and not throughout the iterations. Because in the case of Problem~\eqref{eq:centralProblem}, the cost terms $\V_i(\T{i})$ are not defined for infeasible times $\T{i}\notin \feasT_i$, a remedy needs to be devised. 


We compare two methods based on: (a) introducing soft constraints using slack variables and penalising the $\ell_1$-norm of the constraint violation with a sufficiently high weight and (b) using a feasibility-restoring projection of each iterate, as proposed in~\cite{Zanon2017}. In the following we will call approach (a) \emph{relaxation} and approach (b) \emph{projection}. In~\cite{Zanon2017} the use of projection is preferred to the use of relaxation. However, this choice is not motivated formally and no comparison between the two approaches is made. 

While appealing for its simplicity, approach (a) suffers from one main drawback: the local quadratic approximation evaluated at an infeasible point is dominated by the cost associated with the constraint violation and looses its validity once the problem becomes feasible. This approach can be implemented by replacing constraints~\eqref{eq:localPosOut} in each agent's problem by
\begin{subequations}
	\label{eq:localPosOutSlack}
	\begin{align}	
		p_i(\Tout{i},\vars{i} )-\pout{i} + s_1 - s_2 = 0, \label{eq:localPosOutMod} \\
		s_1 \geq0, \quad s_2 \geq0.
	\end{align}
\end{subequations}

\begin{Proposition}
	Reformulate constraints~\eqref{eq:localPosOut} as~\eqref{eq:localPosOutSlack} and to add the term $\rho s_1+\rho s_2$ to the cost of each local Problem~\eqref{eq:localProblem}. Choose $\rho>\lambda_\mathrm{max}$, i.e. the largest Lagrange multiplier associated with any constraint of the local Problem~\eqref{eq:localProblem} for all feasible in-out times $\T{i}$. Then, the SQP algorithm with linesearch will converge.
\end{Proposition}

\begin{proof}
	The condition $\rho>\lambda_\mathrm{max}$ is necessary in order to make sure that the relaxed problem yields the solution to the original problem whenever the latter has a feasible solution~\cite[Theorem 14.3.1]{Fletcher1987}. The relaxation of the constraints ensures that the local problems are feasible for any in- out-time pair and linesearch ensures a decrease in the merit function. Therefore, for all feasible $t$, the relaxed formulation coincides with the original one, while for all infeasible $t$, linesearch ensures that progress is made towards both feasibility and optimality.
\end{proof}

In order to efficiently implement the projection approach (b), we observe that (i) the constraints~\eqref{eq:agentFeasibility} are simple bounds on the in times $\Tin{i}$ which are always satisfied, and (ii) the constraints on the out times are given by $\Tout{i}\geq\ToutLB{i} \left (\Tin{i}\right )$ and $\Tout{i}\leq\ToutUB{i} \left (\Tin{i}\right )$.
Therefore, the in times $\Tin{i}$ are always feasible and all out times $\Tout{i}$ which become infeasible at iterate $k$ can be projected back onto the set of feasible times by using
\begin{align*}
\mathcal{P}_i\left (\Tout{i}^{(j)} \right ) &:= \left [\Tout{i}^{(j)}\right ]^{\ToutUB{i} \left (\Tin{i}^{(j)}\right )}_{\ToutLB{i} \left (\Tin{i}^{(j)}\right )},
\end{align*}
with $[a]_b^c := \mathrm{max}( \mathrm{min}(a,c),b)$.
Moreover, such a projection can be done easily without extra computations, as the evaluation of ${\ToutUB{i} \left (\Tin{i}^{(j)}\right )}$, ${\ToutLB{i} \left (\Tin{i}^{(j)}\right )}$ is needed for each iterate $j+1$ of the SQP algorithm. An illustration of the projection procedure is given in~\cite{Zanon2017}.
The main drawback of this approach is the partial loss of parallelisability at the agent level, i.e. the two LPs for the computation of the linearisation of constraint~\eqref{eq:agentFeasibility} must be solved before QP~\eqref{eq:localProblem}.


While the projection is simple to perform, the linesearch procedure has to be modified in order to guarantee that the step is a descent direction, provided that $\alpha$ is chosen small enough.
We therefore define a modification of the Armijo condition~\eqref{eq:armijo} using the projection as follows
\begin{align}
\label{eq:proj_armijo}
M( \P (\T{} + \alpha \Delta \T{} )  ) \leq M( \T{} ) + \gamma \mathrm{D} M( \T{} )[\Delta \T{}] \alpha\Delta \T{},
\end{align}
with $\Delta \T{}$ computed by~\eqref{eq:sqp_full_step}. 
As opposed to the standard linesearch approaches, in~\eqref{eq:proj_armijo} we do not backtrack from the projected Newton step. Instead, we backtrack from the non-projected Newton step to obtain the step candidate and project each step candidate.

\begin{Proposition}
	The Armijo condition with projection~\eqref{eq:proj_armijo} is guaranteed to hold, provided that $\alpha$ is chosen small enough.
\end{Proposition}

\begin{proof}
	The proof is provided in~\cite{Zanon2017}.
\end{proof}

\section{Results}
\label{sec:results}

In order to test the algorithm, we have considered vehicles defined by the linear system
\begin{align*}
\dot x = A_\mathrm{c} x + B_\mathrm{c} u, && A_\mathrm{c} = \matr{cc}{0 & 1 \\ 0 & 0}, && B_\mathrm{c} = \matr{c}{0  \\ 1},
\end{align*}
where $u\in [u^\mathrm{lb},u^\mathrm{ub}] \ \mathrm{m/s}^2$ is the (scalar) control and $x=(p,v)$ is the state vector, which includes position $p$ and velocity $v\geq0$. We discretise the system using zero-order hold to obtain
\begin{align*}
x_{k+1} = A x_k + B u_k, && A = \matr{cc}{1 & T_\mathrm{s} \\ 0 & 1}, && B = \matr{c}{0.5 \, T_\mathrm{s}^2 \\ T_\mathrm{s}},
\end{align*}
where $T_\mathrm{s}$ is the sampling time. Each vehicle $i$ has the cost
\begin{align}
	\label{eq:mpc_cost}
	J_i(\vars{i})=\sum_{k=0}^{N} Q_i(v_{i,k}-v_{i,k}^\mathrm{d})^2 + R_iu_{i,k}^2,
\end{align}
where $Q_i$ and $R_i$ are vehicle-specific weights and $v_{i,k}^\mathrm{d}$ is the desired speed of the vehicle at time instant $k$. 

\subsection{The MPC Formulation}

In the previous sections, we have presented the algorithm to compute the in-out times. While in principle one could formulate the MPC problem of each vehicle based on Probelm~\eqref{eq:localProblem}, we preferred to adopt a slightly modified version, where constraints~\eqref{eq:localPosIn}-\eqref{eq:localPosOut} have been relaxed as $p_i(\Tin{i},\vars{i}) \leq \pin{i}$,  $p_i(\Tout{i},\vars{i}) \geq \pout{i}$. Note that this relaxation still enforces the safety-critical collision avoidance constraint.
Moreover, in order to retain feasibility even in the presence of sensor noise, we formulated constraints~\eqref{eq:localPosIn}-\eqref{eq:localPosOut} as soft constraints with exact penalty~\cite{DeOliveira1994,Scokaert1999a}.

\subsection{Simulations}
\label{sec:simulation}

All simulation parameters are specified in Tables~\ref{tab:sim_params} and~\ref{tab:sim_params2}, where the units have been omitted for ease of reading. The intersection is defined by $\pin{i}=0 \ \mathrm{m}$, and $\pout{i}=8 \ \mathrm{m}$ and the control bounds are given by $-u^\mathrm{lb}=u^\mathrm{ub}=2 \ \mathrm{m/s}^2$ for each vehicle $i$.

\begin{table}[t]
	\caption{Simulation Parameters (SI Units)}
	\label{tab:sim_params}
	\vspace{-1em}
	\begin{center}
		\begin{tabular}{c|rrrrrr}
			Vehicle $\#$	   & 1 & 2 & 3 & 4 & 5 & 6 \\
			\toprule
			$Q_i$ 		 & 1 & 1 & 10 & 10 & 1 & 1  \\
			$R_i$ 	 & 1 & 1 & 1 & 1 & 1 & 1  \\
			$v_i^\mathrm{d}$ 	 & 80 & 80 & 65 & 70 & 70 & 60 \\
			$\hat p_0$ 			 		    	 & -55 & -60 & -55 & -70 & -70 & -60  \\
			$\hat v_0$ 		 		    	& 80 & 80 & 65 & 70 & 70 & 60 \\
		\end{tabular}
	\end{center}

\vspace{-1em}
\end{table}

\begin{table}[t]
	\caption{Scenario Crossing Order}
	\label{tab:sim_params2}
	\begin{center}
		\begin{tabular}{l|rrrrrrr}
			$\#$ scenario	   & 1 & 2 & 3 & 4 & 5 & 6 & 7 \\
			\toprule
			\multirow{6}{*}{Order}   & 1 & 1 & 2 & 1 & 1 & 2 & 2 \\
												& 2 & 2 & 1 & 2 & 2 & 1 & 1 \\
												& 6 & 4 & 3 & 5 & 3 & 3 & 3 \\
												& 3 & 3 & 4 & 4 & 4 & 4 & 6 \\
												& 4 & 5 & 5 & 3 & 6 & 6 & 4 \\
												& 5 & 6 & 6 & 6 & 5 & 5 & 5 \\
		\end{tabular}
	\end{center}

	\vspace{-1em}
\end{table}

%

We compared the two feasibility-enforcing approaches by means of numerical simulations. We have solved 5 scenarios using both the projection and the slack-based feasibility-enforcing techniques. In all simulations we used the parameters $\gamma=0.01$, $\beta=0.5$, $\epsilon=10^{-2}$. 

We stress again that evaluating functions $f$ and $h$ constitutes the major computational cost (solving QP~\eqref{eq:localProblem} and the two LPs~\eqref{eq:LP_time}), while the sensitivity information 
$\nabla   f$, $\nabla   h$, $\nabla^2_{\T{}\T{}}{\mathcal{L}}$ 
is much cheaper to compute. Therefore, every linesearch iteration has about the same computational cost as one SQP iteration. On the communication side, every linesearch iteration is a bit cheaper than an SQP iteration, as much less information needs to be communicated.

The results are displayed in Table~\ref{tab:iterations}. It can be seen that in some scenarios relaxation and projection yield comparable results in terms of SQP and linesearch iterates $n_\mathrm{SQP}$ and $n_\mathrm{ls}$ respectively. However, in some scenarios the relaxation approach performs much worse than the projection approach. 
We moreover remark that the relaxation approach never required Hessian regularisations, while the projection approach added some regularisation once in Scenarios $1$ and $4$ and twice in Scenario $7$. 

\begin{table}[t]
	\caption{SQP and backtracking linesearch iterations}
	\label{tab:iterations}
	\vspace{-1em}
	\begin{center}
		\begin{tabular}{ll|rrrrrrr}
			\multicolumn{2}{l}{$\#$ scenario}	   & 1 & 2 & 3 & 4 & 5 & 6 & 7 \\
			\toprule
			\multirow{2}{*}{Projection} 	&$n_\mathrm{SQP}$		 	   & 6 & 10 & 5 & 4 & 7 & 8 & 8 \\
														  &$n_\mathrm{ls}$					& 6 & 10 & 5 & 4 & 7 & 8 & 8 \\
			\midrule
			\multirow{2}{*}{Relaxation} 	&$n_\mathrm{SQP}$			   & 5 & 7 & 9 & 16 & 9 & 8 & 7 \\
														  &$n_\mathrm{ls}$		 		    & 7 & 10 & 15 & 49 & 8 & 16 & 13 \\
		\end{tabular}
	\end{center}
	\vspace{-1em}
\end{table}

\subsection{Experiments}
\label{sec:experiments}
In this subsection, we present the results of experiments performed at the AstaZero proving ground next to Gothenburg, Sweden. A video summarising the experimental tests is available at~\cite{Hult2016a}.
In these experiments, the coordination algorithm was validated in closed loop on a real intersection using three automated Volvo cars (two Volvo S60 sedans and one Volvo XC90 SUV) and vehicle-to-vehicle (V2V) communication from RENDITS~\cite{Rendits2016}.
Each vehicle was further equipped with Real-Time-Kinematic (RTK) capable GPS receivers, inertial sensors, and a computer network consisting of a MicroAutobox 2 real-time computer and a latop, on which the experiment software was executed.
The SQP algorithm was applied in a distributed fashion using the V2V communication, where the SQP subproblems \eqref{eq:qp_k_definition} computing the primal-dual updates were solved at a central computational node. The evaluation of the cost function, constraints and the corresponding sensitivities were done on-board the vehicles by solving the QPs and LPs using HPMPC~\cite{Frison2014}. Due to its higher reliability, the projection approach was chosen for the experimental implementation.

In each experiment the vehicles were controlled from stand-still to a configuration from which a three way collision would occur if the speeds of all vehicles were held constant. 
The coordination was set to start at a predefined time. Short before, the NLP was solved using the proposed distributed algorithm by relying on wireless communication between the vehicles and a central node.

Over the experimental campaign, we successfully completed over $100$  runs with $3$ vehicles and over $50$ runs with $2$ vehicles, and explored various parameter settings. 
However, for the sake of brevity, we present the results of only one run in this paper.
In this run, the vehicles were all controlled to be at $\hat p_0=-200 \ \mathrm{m}$, $\hat v_0=50 \ \mathrm{km/h}$ at time $t_\mathrm{c}=31.3 \ \mathrm{s}$ from the beginning of the run, which was also used as the starting time for the coordination.
Furthermore, for all vehicles, the control objective was to track the desired velocity   $v^\mathrm{d}=50 \ \mathrm{km/h}$,  using the  weights $Q_i=1\ \frac{\mathrm{s}^2}{\mathrm{m}^2}$ and  $R_i=10\ \frac{\mathrm{s}^4}{\mathrm{m}^2}$ and the horizon length $N=200$. 
The control bounds were set to $u^\mathrm{lb}=-3 \ \mathrm{m/s}^2$, $u^\mathrm{ub}=1.6 \ \mathrm{m/s}^2$.
Finally, we used the SQP parameters  $\gamma=0.01$, $\beta=0.5$, $\epsilon=10^{-2}$.

\begin{figure}[t]
	\begin{center}
		\psfrag{KKTstep}[Bc][Bc][0.8][0]{$\|r\|_\infty$, $\|\Delta \T{}\|_\infty$}
		\psfrag{tinout}[Bc][Bc][0.8][0]{$\Tin{},\Tout{}$ [s]}
		\psfrag{t}[Bc][Bc][0.8][0]{$t$ [s]}
		\includegraphics[width=0.45\textwidth]{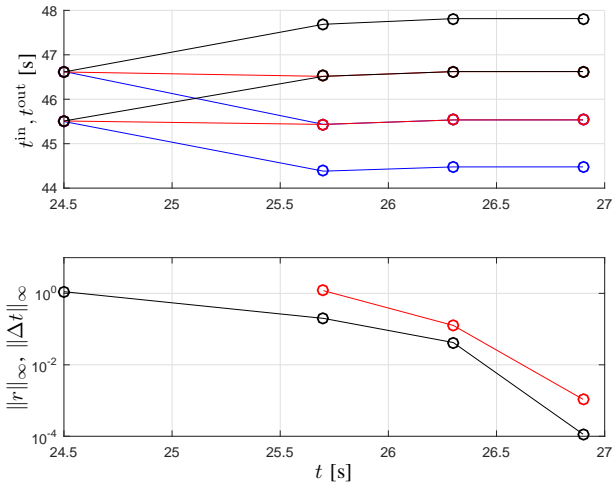}
		\caption{SQP convergence and optimal times: first car in blue, second in red and third in black. Top graph: in-out times for the three vehicles. Bottom graph: norm of the KKT residual $r$ (black) and step $\Delta \T{}$ (red).}
		\label{fig:sqp}
	\end{center}
\end{figure}
The results are displayed in Figure~\ref{fig:sqp}, where one can see the evolution in time of the KKT residual $r$, the step size $\Delta \T{}$ and the solution at the given time. In this particular run, the solver always takes full steps without any projection, which entails that feasibility is obtained after one step while the successive steps improve optimality.
All times are given with respect to the beginning of the scenario, which includes the startup phase to bring the vehicles to the assigned initial configuration from standstill. 
While the SQP algorithm is initialized at $t=24.0 \ \mathrm{s}$, the sensitivities for the first iterate become available only at $t=24.5 \ \mathrm{s}$. To compute the SQP iterates, the central node needs to wait $5$, $10$, $5$ and $5$ sampling instants respectively. These comparatively long iteration times are due to implementation-specific details and could be reduced significantly with a better implementation. In particular, we emphasize that the total computation time, i.e. solving the central SQP subproblems and the vehicle level QPs and LPs, only constitutes  about $3.4\%$ of the total computation time. Moreover, the total time per iterate varies over the iterations. This is due to the unreliability of the wireless communication channel, e.g. some packets are dropped. Therefore, variations would be present even with a more efficient implementation, though with a much smaller magnitude. 



\begin{figure}[t]
	\begin{center}
		\psfrag{p}[Bc][Bc][0.8][0]{$p$ [m]}
		\psfrag{pp}[Bc][Bc][0.8][0]{$p_{50}$ [m]}
		\psfrag{v}[Bc][Bc][0.8][0]{$v$ [km/h]}
		\psfrag{t}[Bc][Bc][0.8][0]{$t$ [s]}
		\includegraphics[width=0.45\textwidth]{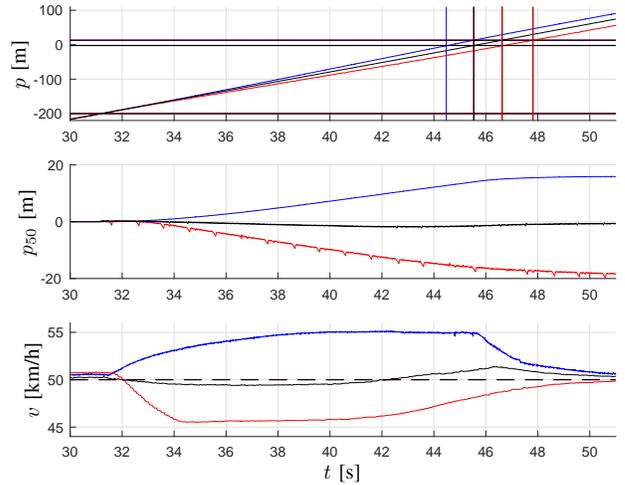}
		\caption{Position trajectories (top graph, a zoom is given in Figure~\ref{fig:traj2}): first car in blue, second in red and third in black. The intersection start and stop is marked by two black lines, the coordination starts when the three vehicles are at $\hat p_0=-200 \ \mathrm{m}$ and the vertical lines mark the in- out-times for each vehicle. Deviation of the position with respect to the constant speed scenario (middle graph). Velocity trajectory (bottom graph): reference velocity $v^\mathrm{d}$ in dashed lines and measured velocities in continuous line.}
		\label{fig:traj}
	\end{center}
\end{figure}

The closed-loop trajectories are displayed in Figure~\ref{fig:traj}. As one can expect, at the solution $\Tout{1}=\Tin{2}$ and $\Tout{2}=\Tin{3}$.
Instead of proceeding at constant velocity, the velocities are adjusted and the vehicles deviate from their reference in order to meet the crossing times. Note that  the second car does not keep a constant velocity but accelerates as it approaches the intersection. This reduces the duration of its permanence inside the intersection, thus allowing the other cars to deviate less from their references and, therefore, to reduce the overall cost. 

Without coordination, all vehicles would proceed from the coordination starting point at a constant velocity $v=50 \ \mathrm{km/h}$.  The difference in position between the coordinated and uncoordinated solutions, here denoted $p_{50}$, is also displayed Figure~\ref{fig:traj}. 
Because we do not provide any reference for the position, after crossing the intersection, the first car is ahead of the position it would have in the absence of the other two cars, whereas the last vehicle lags behind.

The positions shown in Figure~\ref{fig:traj} and Figure~\ref{fig:traj2} were computed based on readings from GPS receivers, and the quality of the signal differs between three cars. In particular, the GPS unit for the third car (red line in the plots) systematically provided unreliable measurements with a frequency of $1 \ \mathrm{Hz}$.
Due to such measurement errors and the mismatch between the prediction model and the physical vehicles, the constraints enforcing the in and out times are always violated by the closed loop system.
We define the time violations as $\delta \Tin{i}$, $\delta \Tout{i}$, where a positive sign means that the actual time was larger than planned, such that $\delta \Tin{i}\geq 0$ corresponds to a safe configuration (the vehicle entered late), while $\delta \Tout{i}\geq0$ corresponds to a dangerous configuration (the vehicle left late). For the position we kept the convention of the constraint definition, such that $p_i(\Tin{i},\vars{i})-\pin{i}\leq0$ is safe, while $p_i(\Tout{i},\vars{i})-\pout{i}\leq 0$ is dangerous. 
The estimated violation of the in-out times for the  experiment run is given in Table~\ref{tab:time_pos_error}. 
It should be noted that the values given in Table~\ref{tab:time_pos_error} are taken from the position signal displayed in Figure~\ref{fig:traj2}. For vehicle 2 and 3 this signal is rather noisy and not very reliable, hence the large violations. We remark that the observed constraint violations can be easily accounted for in the proposed setup by constraint tightening, i.e. by enlarging the intersection definition. Because the constraint has the dimension of a distance, the tightening procedure can be directly related to positioning errors.


The analysis of the obtained experimental results suggests that the proposed decomposition of the control scheme has the following desirable properties: (a) it does not need to rely on efficient and reliable communication; (b) even if not implemented in the most efficient way, the algorithm converges fast enough to allow for a real-time implementation; (c) after the solution has been computed, communication is not required anymore and the vehicles can be safely controlled in a decoupled fashion; (d) sporadic wrong sensor readings are natively handled by the algorithm.

After a thorough experimental campaign in a safe configuration, i.e. driving the vehicles on parallel lanes, we ran some experiments on a real crossing scenario. For safety reasons, we added $5 \ \mathrm{m}$ to the intersection length. An aerial picture is displayed in Figure~\ref{fig:crossing}, where we highlighted the intersection by a red box. In the real crossing experiment, we obtained results comparable to those obtained in the parallel lane configuration.

\begin{figure}[t]
	\begin{center}
		\psfrag{p}[Bc][Bc][0.8][0]{$p$ [m]}
		\psfrag{pp}[Bc][Bc][0.8][0]{$p_{50}$ [m]}
		\psfrag{v}[Bc][Bc][0.8][0]{$v$ [km/h]}
		\psfrag{t}[Bc][Bc][0.8][0]{$t$ [s]}
		\includegraphics[width=0.45\textwidth]{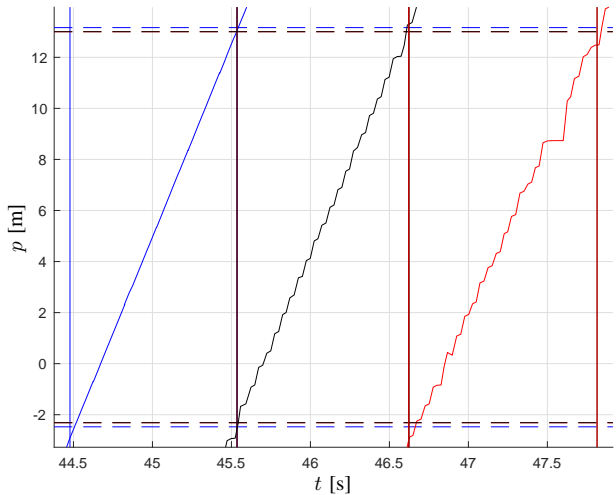}
		\caption{Zoom of the position trajectories from Figure~\ref{fig:traj} in the intersection area. The horizontal dashed lines mark the beginning and the end of the intersection for each vehicle; the red line is superposed to the black one, as the two corresponding cars are identical. The vertical lines mark the in and out times sent to each car by the central node.}
		\label{fig:traj2}
	\end{center}
\end{figure}

\section{Conclusions and Future Research}
\label{sec:conclusions}

In this paper, we have proposed a distributed control scheme for optimal vehicle coordination at intersections. We have developed an ad-hoc optimisation algorithm based on primal decomposition and a feasibility-enforcing projection. We have compared this approach versus a the relaxation feasibility-enforcing approach in simulations which have indicated a lower reliability of the latter.

We have tested our control approach in experiments using real cars on a test track. Though our implementation of the algorithm on the real-time platform was simple and far from optimal, the experimental results have demonstrated the applicability and real-time feasibility of our approach. Moreover, the sporadic faulty sensor readings and the unreliable communication have both demonstrated some degree of robustness of our control approach.

Future work will consider improved approaches that (a) update the in and out times in a closed-loop fashion, (b) solve the ordering problem, i.e. the full mixed-integer problem.

\begin{figure}[t]
	\begin{center}
\begin{tikzpicture}
	\node[anchor=south west,inner sep=0] at (0,0) {\includegraphics[width=\linewidth]{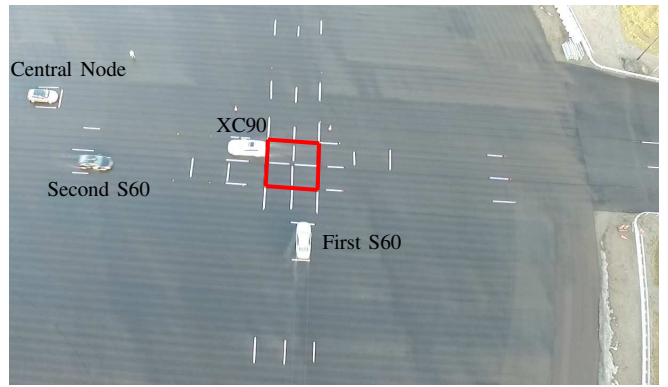}};
	\node at (0.8,4.2){\footnotesize Central Node};
	\node at (3.1,3.45){\footnotesize XC90};
	\node at (1.2,2.6){\footnotesize Second S60};
	\node at (4.7,1.9){\footnotesize First S60};
	\draw[line width=1.5pt,-,red] (3.42,2.64) -- (4.11,2.59);
	\draw[line width=1.5pt,-,red] (4.11,2.59) -- (4.13,3.21);
	\draw[line width=1.5pt,-,red] (4.13,3.21) -- (3.45,3.26);
	\draw[line width=1.5pt,-,red] (3.45,3.26) -- (3.42,2.64);
	
%
\end{tikzpicture}
		\caption{Aerial picture of the crossing experiment.}
		\label{fig:crossing}
	\end{center}
\end{figure}

\bibliographystyle{plain}
\bibliography{../../bibliography}

\begin{table}[t]
	\caption{In-out time and position errors (SI units)}
	\label{tab:time_pos_error}
	\begin{tabular}{r|cccc}
		Car \# & $\delta \Tin{i}$ & $\delta \Tout{i}$ & $p_i(\Tin{i},\vars{i})-\pin{i}$ & $p_i(\Tout{i},\vars{i})-\pout{i}$ \\
		\toprule
		1  & $0.0287$   & $\phantom{-} 0.0078$  & $-0.4363$  &                    $-0.1213$ \\
		2 & $0.0077$    &                   $-0.0156$  & $-0.2461$  &  $\phantom{-} 0.2914$ \\
		3 & $0.0480$    & $\phantom{-}0.0257$  & $-0.6610$   &                    $-0.5217$ \\
	\end{tabular}
\end{table}

\end{document}